\documentclass[final,11pt,english]{amsart}
\usepackage{amsfonts,amsmath,amsthm,amscd,amssymb,latexsym,amsbsy,pb-diagram,cite,caption2}
\usepackage{amsfonts}
\usepackage[dvips]{graphicx}
\usepackage[mathscr]{eucal}
\usepackage{ifthen}
\usepackage{graphicx}

\newcounter{minutes}
\divide\time by 60
\newcounter{hours}
\multiply\time by 60 \addtocounter{minutes}{-\time}

\usepackage[notref,notcite]{showkeys}

\newtheorem{theorem}[equation]{Theorem}
\newtheorem{lemma}[equation]{Lemma}
\newtheorem{proposition}[equation]{Proposition}
\newtheorem{example}[equation]{Example}
\newtheorem{corollary}[equation]{Corollary}
\newtheorem{remark}[equation]{Remark}
\newtheorem{definition}[equation]{Definition}

\numberwithin{equation}{section}

 \DeclareMathOperator{\card}{card}
\DeclareMathOperator{\diam}{diam}

\begin{document}

\def\thefootnote{}
\footnotetext{ \texttt{\tiny File:~\jobname .tex,
          printed: \number\year-\number\month-\number\day,
          \thehours.\ifnum\theminutes<10{0}\fi\theminutes}
} \makeatletter\def\thefootnote{\@arabic\c@footnote}\makeatother

\title{\bf Metrization of weighted graphs}

\author{\bf Oleksiy Dovgoshey, Olli Martio and Matti Vuorinen}

\date{}

\maketitle
\begin{abstract}
We find a set of necessary and sufficient conditions under which the
weight $w:E\to\mathbb R^+$ on the graph $G=(V,E)$ can be extended to
a pseudometric $d:V\times V\to\mathbb R^+$. If these conditions hold
and $G$ is a connected graph, then the set $\mathfrak M_w$ of all
such extensions is nonvoid and the shortest-path pseudometric $d_w$
is the greatest element of $\mathfrak M_w$ with respect to the
partial ordering $d_1 \leqslant d_2$ if and only if $d_1(u,v)
\leqslant d_2(u,v)$ for all $u,v\in V$. It is shown that every
nonvoid poset $(\mathfrak M_w,\leqslant)$ contains the least element
$\rho_{0,w}$ if and only if $G$ is a complete $k$-partite graph with
$k\geqslant 2$ and in this case the explicit formula for computation
of $\rho_{0,w}$ is obtained.
\end{abstract}

\section{Introduction}

Recall the basic definitions that we adopt here. A graph $G$ is an
ordered pair $(V,E)$ consisting of a set $V=V(G)$ of {\it vertices}
and a set  $E=E(G)$ of {\it edges}. In this paper we study the {\it
simple} graphs which are finite, $\card(V)<\infty$, or infinite,
$\card(V)=\infty$. Since our graph $G$ is simple we can identify
$E(G)$ with a set of two-element subsets of $V(G)$, so that each
edge is an unordered pair of distinct vertices. As usual we suppose
that $V(G)\cap E(G)=\emptyset$. The edge $e=\{u,v\}$ is said to {\it
join} $u$ and $v$, and the vertices $u$ and $v$ are called {\it
adjacent} in $G$. The graph $G$ is {\it empty} if no two vertices
are adjacent, i.e. if $E(G)=\emptyset$. We use the standard
definitions of the {\it path}, the {\it cycle}, the {\it subgraph}
and {\it supergraph}, see, for example, \cite[p.~4, p.~40]{BM}. Note
only that all paths and cycles are finite and simple graphs.

The following, basic for us, notion is a {\it weighted graph}
$(G,w)$, i.e., a graph $G=(V,E)$ together with a weight
$w:E\to\mathbb R^+$ where $\mathbb{R}^+=[0,\infty)$. If $(G,w)$ is a
weighted graph, then for each subgraph $F$ of the graph $G$  define
\begin{equation}\label{e1}
w(F):=\sum_{e\in E(F)}w(e).
\end{equation}
The last sum may be equal $+\infty$ if $F$ is infinite.

Recall also that a {\it pseudometric} $d$ on the set $X$ is a
function $d:X\times X\to\mathbb R^+$ such that $d(x,x)=0,\
d(x,y)=d(y,x)$ and $d(x,y)\leq d(x,z)+d(z,y)$ for all $x,y,z\in X$.
The pseudometric $d$ on $X$ is a {\it metric} if, in addition,
$$
(d(x,y)=0)\Rightarrow(x=y)
$$
for all $x,y\in X$. Using a pseudometric $d$ on the set $V$ of
vertices of the graph $G=(V,E)$ one can simply define a weight
$w:E\to\mathbb R^+$ by the rule
\begin{equation}\label{e2}
w(\{u,v\}):=d(u,v)
\end{equation}
for all edges $\{u,v\}\in E(G)$. The correctness of this
definition follows from the symmetry of  $d$.

A legitimate  question to raise in this point is whether there
exists a pseudometric $d$ such that the given weight $w:E\to\mathbb
R^+$ is produced as in \eqref{e2}. If yes, then we say that $w$ is a
{\it metrizable weight}.

\begin{figure}[h]
\includegraphics[width=\textwidth,keepaspectratio]{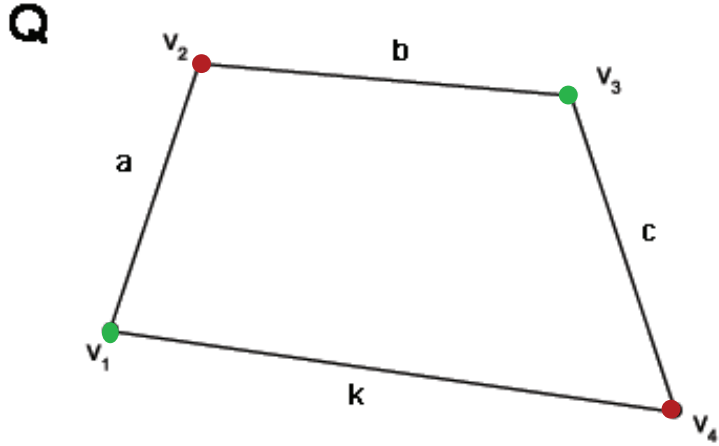}
\caption{Here $(Q,w)$ is a weighted quadrilateral with
$V(Q)=\{v_1,v_2,v_3,v_4\}$,
$E(Q)=\{\{v_1,v_2\},\{v_2,v_3\},\{v_3,v_4\},\{v_4,v_1\}\}$ and
\newline
$w(\{v_1,v_2\})=a,\quad w(\{v_2,v_3\})=b,\quad
w(\{v_3,v_4\})=c,\quad w(\{v_4,v_1\})=k. $ }
\end{figure}

The above formulated question seems to be converse for the question
of embeddings of metrics into weighted graphs. (In the standard
terminology one says about the {\it realization} of metric spaces by
graphs.) This topic is rich and has many applications in various
areas, such as psychology, phylogenetic analysis and recent
applications from the field of computer science. Some results and
references in this direction can be found in \cite{DL} and \cite{L}.

If $(G,w)$ is a weighted graph with metrizable $w$, then we shall
denote by $\mathfrak{M}_w$ the set of all pseudometrics $d:V\times
V\to \mathbb{R}^+$ such that
$$
d(v_i,v_j)=w(\{v_i,v_j\})
$$
for all $\{v_i,v_j\}\in E(G)$.

The starting point of our considerations is the following Model
Example.

\begin{theorem}[Model Example]\label{t0}
Let $(Q,w)$ be a weighted graph depicted by Figure 1. The weight $w$
on the graph $Q$ is metrizable if and only if
\begin{equation}\label{e1*}
2\max\{a,b,c,k\}\leqslant a+b+c+k.
\end{equation}

If $w$ is metrizable, then for each $d\in\mathfrak{M}_w$ we have the
double inequalities
$$
\max\{|b-c|,|a-k|\}\leqslant d(v_2,v_4)\leqslant\min\{b+c,a+k\}
$$
and
\begin{equation}\label{e2*}
\max\{|a-b|,|c-k|\}\leqslant d(v_1,v_3)\leqslant\min\{a+b,c+k\}.
\end{equation}

Conversely, if $p$ and $q$ are real numbers such that
$$
\max\{|b-c|,|a-k|\}\leqslant p\leqslant\min\{b+c,a+k\}
$$
and
\begin{equation}\label{e3*}
\max\{|a-b|,|c-k|\}\leqslant q\leqslant\min\{a+b,c+k\},
\end{equation}
then there is $d\in\mathfrak{M}_w$ with
$$
d(v_2,v_4)=p\qquad,\qquad d(v_1,v_3)=q.
$$
\end{theorem}

This theorem was proved in \cite{DP} and used there as a base to
finding of extremally Ptolemeic and extremally non-Ptolemeic metric
spaces. The results of the present paper generalize the Model
Example to the case of arbitrary (finite or infinite) weighted
graphs $(G,w)$.
\begin{itemize}
\item[\rm --] Theorem \ref{t1} gives  necessary and sufficient
conditions under which a weight $w$ is metrizable. The key point
here is an extension of inequality \eqref{e1*} to an arbitrary cycle
$C\subseteq G$.

\item[\rm --] Proposition \ref{p7} claims that for connected $G$ and
metrizable $w$ the shortest-path pseudometric $d_w$ belongs to
$\mathfrak{M}_w$ and that this pseudometric is the greatest element
of $\mathfrak{M}_w$. The reader can observe that the right-side in
double inequalities \eqref{e2*} and \eqref{e3*} are, in fact,
$d_w(v_2,v_4)$ and $d_w(v_1,v_3)$.

\item[\rm --] Theorem \ref{t21} shows that the least pseudometric in
$\mathfrak{M}_w$, (see the left-side in \eqref{e2*}, \eqref{e3*})
exists for each metrizable $w$ if and only if $G$ is a complete
$k$-partite graph with $k\geqslant 2$.

\item[\rm --] In Theorem \ref{t40} we show that for complete
$k$-partite graphs $G$ with $k\geqslant 2$ and with the cardinality
of partitions $\leqslant 2$ we have the analog of the last part of
the Model Example: a symmetric function $f:V\times V\to
\mathbb{R}^+$ belongs to $\mathfrak{M}_w$ if and only if it "lies
between" the greatest element of $\mathfrak{M}_w$ and the least one.
\end{itemize}

Moreover in Theorem \ref{t13} we describe the structure of connected
graphs $G$ which admit strictly positive metrizable weights $w$ such
that $\mathfrak{M}_w$  does not contain any metrics.

\section{Embeddings of weighted graphs\\ into pseudometric spaces}

Let $(G,w)$ be a weighted graph and let $u,v$ be vertices belonging
to a connected component of $G$. Let us denote $\mathcal
P_{u,v}=\mathcal P_{u,v}(G)$ the set of all paths joining  $u$ and
$v$ in $G$. Write
\begin{equation}\label{e3}
d_w(u,v):=\inf\{w(F):F\in\mathcal P_{u,v}\}
\end{equation}
where $w(F)$ is the weight of the path $F$, see formula \eqref{e1}.
It is well known for the connected graph $G$ that the function $d_w$
is a pseudometric on the set $V(G)$. This pseudometric will be
termed as the {\it weighted shortest-path pseudometric}. It
coincides with the usual path metric if $w(e)=1$ for all $e\in
E(G)$.

\begin{theorem}\label{t1}
Let $(G,w)$ be a weighted graph. The following statements are
equivalent.
\begin{itemize}

\item[\rm(i)] The weight $w$ is metrizable.

\item[\rm(ii)] The equality
\begin{equation}\label{e4}
w(\{u,v\})=d_w(u,v)
\end{equation}
holds for all $\{u,v\}\in E(G)$.

\item[\rm(iii)] For every cycle $C\subseteq G$ we have the
inequality
\begin{equation}\label{e5}
2\max_{e\in E(C)}w(e)\leq w(C).
\end{equation}
\end{itemize}
\end{theorem}

It seems to be interesting to have conditions under which the set
$\mathfrak{M}_w$ contains some metrics of a special type. In
particular: What are restrictions on the weight $w$ guaranteeing the
existence of ultrametrics (or pseudoultrametrics) in the set
$\mathfrak{M}_w$?

\begin{remark}\label{r4}
If $C$ is a 3-cycle, then \eqref{e5} turns to the symmetric form
$$
2\max\{w(e_1),w(e_2),w(e_3)\}\leq w(e_1)+w(e_2)+w(e_3)
$$
of the triangle inequality. Thus \eqref{e5} can be considered as a
``cyclic generalization'' of this inequality.
\end{remark}
\begin{remark}\label{r5}
Theorem \ref{t1} evidently holds if $G$ is the {\it null graph},
i.e. if $V(G)=\emptyset$. In this case the related metric space
$(V,d)$ is empty.
\end{remark}
\begin{proof}[Proof of Theorem \ref{t1}]
{\bf(i)$\Rightarrow$(ii)} Suppose that there is a pseudometric
$\rho$ on $V$ such that
$$
w(\{u,v\})=\rho(u,v)
$$
for each $\{u,v\}\in E(G)$. Then for every sequence of points
$v_1,\dots v_n,\ v_1=u$ and $v_n=v,\ v_i\in V,\ i=1,\dots,n$, the
triangle inequality implies
$$
w(\{u,v\})=\rho(v_1,v_n)\leq\sum_{i=1}^{n-1}\rho(v_i,v_{i+1}).
$$
Consequently for paths $F\subseteq G$ joining   $u$ and $v$ the
inequality
$$
w(\{u,v\})\leq w(F)
$$
holds. Passing in the last inequality to the infimum over the set
$\{w(F): F\in\mathcal P_{u,v}\}$ we obtain
\begin{equation}\label{*}
\rho(u,v)=w(\{u,v\})\leq d_w(u,v),
\end{equation}
see \eqref{e3}. The converse inequality $w(\{u,v\})\geq d_w(u,v)$
holds because the path $(u=v_1,v_2=v)$ belongs to $\mathcal
P_{u,v}$.

{\bf(ii)$\Rightarrow$(iii)} Suppose statement (ii) holds. Let $C$
be an arbitrary cycle in $G$ and let $\{u,v\}\in E(C)$ be an edge
for which
\begin{equation}\label{e8}
w(\{u,v\})=\max_{e\in E(C)}w(e).
\end{equation}
Deleting the edge $\{u,v\}$ from the cycle $C$ we obtain the path
$F:=C\setminus\{u,v\}$ joining  the vertices $u$ and $v$. Using
equalities \eqref{e3}, \eqref{e4} and \eqref{e8} we conclude that
\begin{equation}\label{e9}
\max_{e\in E(C)}w(e)=d_w(u,v)\leq w(F).
\end{equation}
Since $w(F)=w(C)-w(\{u,v\})$, \eqref{e5} follows from \eqref{e9}.

{\bf(iii)$\Rightarrow$(i)} Suppose (iii) is true. If $G$ is a
connected graph, then we can equip $G$ by the weighted shortest-path
pseudometric $d_w$, so it is sufficient to show that
$d_w\in\mathfrak{M}_w$. Let $\{u,v\}\in E(G)$. In the case where there
is no cycle $C\subseteq G$ such that $\{u,v\}\in E(C)$ the path
$(u=v_1,v_2=v)$ is the unique path joining $u$ and $v$. Hence, in
this case, equality \eqref{e4} follows from \eqref{e3}. Let
$P=(u=v_1,\dots,v_{k+1}=v)$ be an arbitrary $k$-path, $k\geq 2$,
joining $u$ and $v$. Then $C:=(u=v_1,\dots,v_{k+1},v_{k+2}=u)$ is a
$k+1$-cycle with $\{u,v\}\in E(C)$. Hence by \eqref{e5} we have
$$
2w(\{u,v\})\leq 2\max_{e\in E(C)}w(e)\leq w(C)=w(P)+w(\{u,v\}).
$$
This implies the inequality $ w(\{u,v\})\leq w(P) $ for all
$P\in\mathcal P_{u,v}$. Consequently $ w(\{u,v\})\leq d_w(u,v). $
The converse inequality is trivial. Thus if $G$ is connected, then
$d_w\in\mathfrak{M}_w$.

Consider now the case of disconnected graph $G$. Let
$\{G_i:i\in\mathcal I\}$ be the set of all components of $G$ and let
$\{v_i^*:i\in\mathcal I\}$ be the subset of $V(G)$ such that
$$
v_i^*\in V(G_i)
$$
for each $i\in\mathcal I$. We choose an index $i_0\in\mathcal I$ and
fix nonnegative constants $a_i,\ i\in\mathcal I$ such that
$a_{i_0}=0$. Let us define the function $\rho:V(G)\times V(G)\to
\mathbb{R}^+$ as
\begin{equation}\label{e11}
\rho(u,v)=d_{w_i}(u,v)
\end{equation}
if $u$ and $v$ lie in the same component $G_i$ and as
\begin{equation}\label{e12}
\rho(u,v)=a_i+a_j+d_{w_i}(u,v_i^*)+d_{w_j}(v,v_j^*)
\end{equation}
if $u\in G_i$ and $v\in G_j$ with $i\neq j$. Here $w_i$ is the
restriction of $w$ on the set $E(G_i)$ and $d_{w_i}$ is the weighted
shortest-path pseudometric corresponding to the weight $w_i$. It is
easy to see that
\begin{equation}\label{e12*}
a_i=\rho(v_i^*,v_{i_0}^*)
\end{equation}
for all $i\in\mathcal I$.

It follows directly from \eqref{e11} and \eqref{e12} that $\rho$ is
a pseudometric on $V(G)$ and $\rho\in\mathfrak{M}_w$.
\end{proof}

\begin{remark}\label{r6}
To obtain the pseudometric $\rho$ described by formulas \eqref{e11},
\eqref{e12} we can consider the supergraph $G^*$ of $G$ such that
$V(G^*)=V(G)$ and
$$
E(G^*)=E(G)\cup\{\{v_i^*,v_{i_0}^*\}:i\in\mathcal
I\setminus\{i_0\}\},
$$
see Fig.~2. Then $G^*$  is a connected graph with the same set of
cycles as in $G$ and all edges $\{v^*_i,v^*_{i_0}\}$ are bridges of
$G^*$. Now we can extend the weight $w:E(G)\to\mathbb R^+$ to a
weight $w^*:E(G^*)\to\mathbb R^+$ by the rule:
$$
w^*(\{u,v\}):=\begin{cases} w(\{u,v\})&\text{if }\{u,v\}\in E(G)\\
a_i&\text{if }\{u,v\}=\{v_i^*,v_{i_0}^*\},\ i\in\mathcal
I\setminus\{i_0\}.
\end{cases}
$$
It can be shown that the pseudometric $\rho$ is simply the weighted
shortest-path pseudometric with respect to the weight $w^*$.
\end{remark}

\begin{figure}[h]
\includegraphics[width=\textwidth,keepaspectratio]{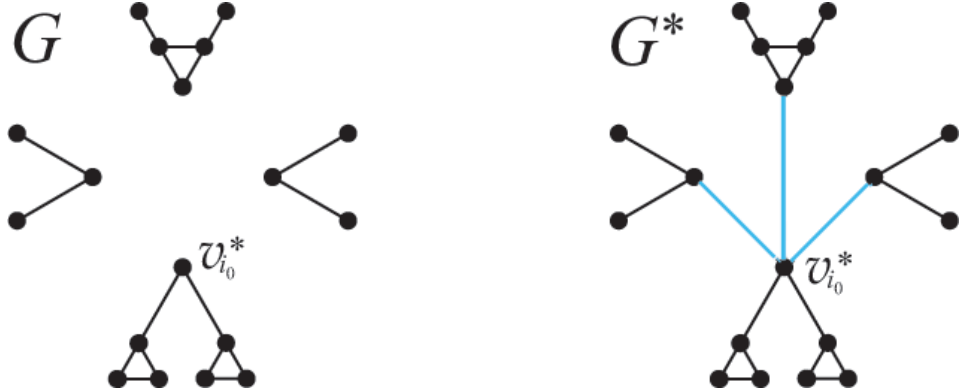}
\caption{Inclusion of the disconnected $G$ in the connected $G^*$.
There are no new cycles in $G^*$.}
\end{figure}

Let $w_1$ and $w$ be two weights with the same underlieing graph
$G$. Suppose the weight $w_1$ is metrizable. What are condition
under which the weight $w$ is also metrizable?

To describe such type conditions we recall the definition of a {\it
bridge}.

\begin{definition}\label{br.d}
Let $G$ be a graph and let $e_0\in E(G)$. For a connected $G$, $e_0$
is a bridge of $G$, if $G-e_0$ is a disconnected graph. If $G$ is
disconnected and $G_0$ is the connected component of $G$ such that
$e_0\in E(G_0)$, then $e_0$ is a bridge of $G$, if $G_0-e_0$ is
disconnected.
\end{definition}

Above we denote by $G-e_0$ the {\it edge-deleted} subgraph of $G$,
see, for example, \cite[p. 40]{BM}.

For weights $w_1$ and $w_2$ on $E(G)$ define a set $w_1\Delta
w_2\subseteq E(G)$ as
$$
w_1\Delta w_2 = \{e\in E(G): w_1(e)\neq w_2(e)\}.
$$

\begin{proposition}\label{br.pr}
Let $(G,w_1)$ be a weighted graph with a metrizable $w_1$ and let
$E_1$ be a subset of $E(G)$. The following statements are
equivalent.
\begin{itemize}

\item[\rm(i)] All weights $w:E(G)\to\mathbb{R}^+$ with $w_1\Delta w\subseteq E_1$ are metrizable.

\item[\rm(ii)] Each element $e\in E_1$ is a bridge of the graph $G$.
\end{itemize}
\end{proposition}

\begin{lemma}\label{br.l}
An edge $e\in E(G)$ is a bridge if and only if $e$ is not in $E(C)$
for any cycle $C\subseteq G$.
\end{lemma}

This lemma is known for the finite graphs $G$, see \cite[Theorem
2.3]{TH}. The proof for infinite $G$ is completely analogous, so
we omit it here.

\begin{proof}[\it Proof of Proposition \ref{br.pr}]
The implication (ii)$\Rightarrow$(i) follows directly from
condition (iii) of Theorem \ref{t1} and Lemma \ref{br.l}.
Conversely,  if some $e_0\in E_1$ is not bridge, then by Lemma
\ref{br.l} there is a cycle $C_0\subseteq G$ such that $e_0\in
E(C_0)$. Let us define the weight $w_0:E(G)\to\mathbb{R}$,
$$
w_0(e)=\begin{cases} w_1(e) & \text{ if } \quad e\neq e_0\\
1+w(C_0)-w(e_0) & \text{ if } \quad e=e_0.
\end{cases}
$$
Then we have $w_1\Delta w_0=\{e_0\}$ and
\begin{multline*}
2w_0(e_0)=2+2w_1(C_0)-2w_1(e_0)>
\\
(w_1(C_0)-w_1(e_0)+1)+(w_1(C_0)-w_1(e_0))=w_0(C_0).
\end{multline*}
It is clear, that $w_1\Delta w_0\subseteq E_1$ but, by Theorem
\ref{t13}, the weight $w_0$ is not metrizable. Thus the implication
(i)$\Rightarrow$(ii) follows.
\end{proof}

Recall that {\it  acyclic} graphs are usually called the {\it
forests}. Lemma \ref{br.l} implies that a graph $G$ is a forest if
and only if all $e\in E(G)$ are bridges of $G$. Hence as a
particular case of Proposition \ref{br.pr} we obtain

\begin{corollary}\label{c6}
The following conditions are equivalent for every graph $G$.
\begin{itemize}

\item[\rm(i)] $G$ is a forest.

\item[\rm(ii)]  Every weight $w:E(G)\to
\mathbb R^+$ is metrizable.
\end{itemize}
\end{corollary}

Our final corollary shows that the property of a weight to be
metrizable is local.

\begin{corollary}\label{c7*}
Let $(G,w_G)$ be a weighted graph. The weight $w_G$ is metrizable if
and only if the restrictions $w_H=w_G|_{E(H)}$ are metrizable for
all finite subgraphs $H$ of the graph $G$.
\end{corollary}

\section{Maximality of the weighted \\shortest-path pseudometric}

Let $G$ be a graph and let $w$ be a metrizable weight on $E(G)$.
Recall that $\mathfrak M_w$ is the set of all pseudometrics $\rho$
on $V(G)$ satisfying the restriction
\begin{equation}\label{e14}
\rho(u,v)=w(\{u,v\})
\end{equation}
for each $\{u,v\}\in E(G)$. Let us introduce the ordering relation
$\leqslant$ on the set $\mathfrak M_w$ as
\begin{equation}\label{e14*}
(\rho_1\leqslant\rho_2)\quad \text{if and only if}\quad
(\rho_1(u,v)\leqslant\rho_2(u,v))
\end{equation}
for all $u,v\in E(G)$. A reasonable question to ask is whether it is
possible to find the greatest and least elements of the partially
ordered set $(\mathfrak M_w,\leqslant)$.

In the present section we show that the shortest-path pseudometric
$d_w$ is the greatest element in $(\mathfrak M_w,\leqslant)$ for a
connected $G$ and apply this result to the search of metrics in
$\mathfrak M_w$. The existence of the least element of the poset
$(\mathfrak M_w,\leqslant)$ will be discussed in Section 4.

\begin{proposition}\label{p7}
Let $(G,w)$ be a nonempty weighted graph with a metrizable weight
$w$. If $G$ is connected then the weighted shortest-path
pseudometric $d_w$ belongs to $\mathfrak{M}_w$ and this pseudometric
is the greatest element of the poset $(\mathfrak{M}_w,\leqslant)$,
i.e., the inequality
\begin{equation}\label{e15}
\rho\leqslant d_w
\end{equation}
holds for each $\rho\in\mathfrak{M}_w$. Conversely, if the poset
$(\mathfrak{M}_w,\leqslant)$ contains the greatest element, then $G$
is connected.
\end{proposition}
\begin{proof}
In fact, for connected $G$, the membership relation $d_w\in\mathfrak
M_w$ was justified in the third part of the proof of
Theorem~\ref{t1}. To prove \eqref{e15} see \eqref{*}.

If $G$ is disconnected and some vertices $u,v$ lie in distinct
components, $u\in G_i,\ v\in G_j$, then letting $a_i,a_j\to+\infty$
in \eqref{e12} we obtain
$$
\sup\{\diam_\rho(A):\rho\in\mathfrak M_w\}=\infty
$$
for the two-element set $A=\{u,v\}$. Thus the poset $(\mathfrak M_w,
\leqslant)$ does not contain the greatest element.
\end{proof}

\begin{remark}\label{r8}
If $G$ is a disconnected graph and $u,v$ belong to distinct
connected components of $G$, then  according to \eqref{e3} we can
put
$$
d_w(u,v)=+\infty
$$
as for the infimum over the empty set. Under this agreement, the
weighted shortest-path pseudometric is also "the greatest element"
of $\mathfrak{M}_w$ for the disconnected graphs $G$.
\end{remark}

Recall that connected acyclic graphs are called the trees, so that
each tree is a connected forest. The last proposition and Corollary
\ref{c6} imply

\begin{corollary}\label{c11*}
A graph $G$ is a tree if and only if each weight
$w:E(G)\to\mathbb{R}^+$ is metrizable and the inequality
$$
\sup\{\diam_\rho(A):\rho\in\mathfrak M_w\}<\infty
$$
holds for every
finite $A\subseteq V(G)$.
\end{corollary}

Let $(G,w)$ be a weighted graph with a metrizable $w$. Then each
$\rho\in\mathfrak M_w$ is a {\it pseudometric} satisfying
\eqref{e14}. We ask under what conditions does a {\bf metric}
$\rho\in\mathfrak M_w$ exist. To this end it is necessary for the
weight $w:E(G)\to\mathbb R^+$ to be {\it strictly positive } in the
sense that $w(e)>0$ for all $e\in E(G)$. This trivial condition is
also sufficient for the graphs with the vertices of {\it finite
degrees}. Here, as usual, by the degree of a vertex $v$ we
understand the cardinal number of edges incident with $v$. More
generally we have

\begin{corollary}\label{c12}
Let $(G,w)$ be a weighted graph such that each connected component
of $G$ contains at most one vertex of infinite degree. If the weight
$w$ is metrizable and strictly positive, then there is a metric
$\rho\in\mathfrak M_w$.
\end{corollary}

\begin{proof}
Suppose $G$ is connected and $w$ is strictly positive and
metrizable. Let $\{u,v\}\notin E(G)$. Without loss of generality we
may suppose that the edges of $G$ which are incident with $u$ form
the finite set $\{e_1,\ldots, e_n\}$. Then the inequalities
$$
w(F)>\underset{1\leqslant i\leqslant n}{\min}w(e_i)>0
$$
holds for each path $F\in\mathcal{P}_{u,v}$. Thus $d_w(u,v)>0$ for
every pair of distinct $u,v\in V(G)$. It still remains to note that
$d_w\in\mathfrak M_w$ by Proposition~\ref{p7}.

For the case of disconnected $G$ we can obtain the desirable metric
$\rho\in\mathfrak M_w$ using \eqref{e11} and \eqref{e12} with
strictly positive $a_i,a_j$.
\end{proof}

\begin{figure}[h]
\centering
\includegraphics[keepaspectratio,width=10cm]{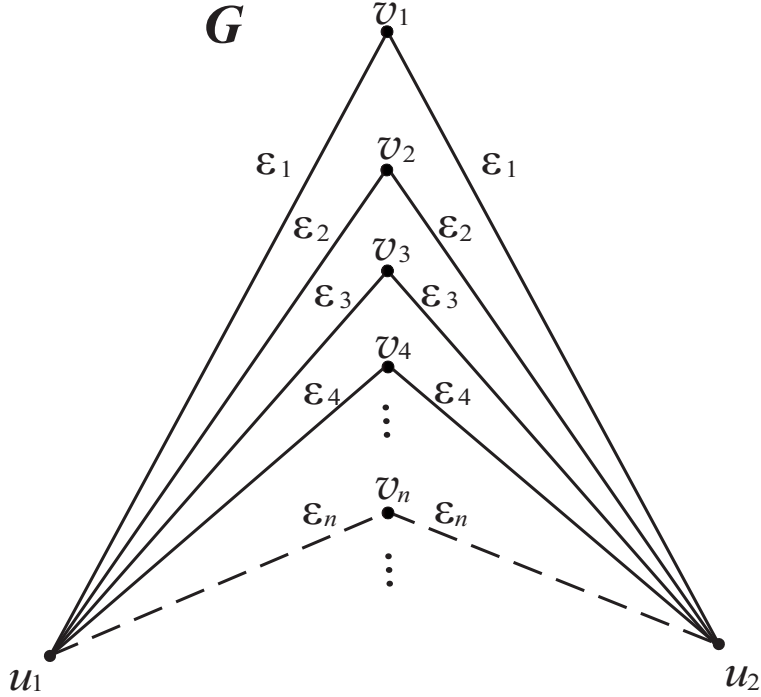}
\caption{The infinite $G$ with metrizable $w$ having no metrics in
$\mathfrak M_w$.}
\end{figure}

\begin{remark}\label{r11}
The main point of the previous proof is the following: If there is
a metric $\rho\in\mathfrak M_w$, then the weighted shortest-path
pseudometric is also a metric.
\end{remark}
The following example shows that the conclusion of Corollary
\ref{c12} is, generally speaking, false for connected graphs $G$,
containing at least two vertices of infinite degree.

\begin{example}\label{ex12}
Let $(G,w)$ be the infinite weighted graph depicted by Fig.~3 where
$\varepsilon_n=w(\{v_n,u_1\})=w(\{v_n,u_2\}),$ are  real numbers
such that
$$
\lim_{n\to\infty}\varepsilon_n=0
$$
and $\varepsilon_n>\varepsilon_{n+1}>0$ for each $n\in\mathbb N$.
Each cycle $C$ of $G$ is a quadrilateral of the form
$u_1,v_n,u_2,v_m,u_1$ with $n\ne m$. Since $C$ has two distinct
edges of the maximal weight, inequality \eqref{e5} holds. By
Theorem~\ref{t1} it means that $w$ is metrizable. Letting
$n,m\to\infty$ and using formula \eqref{e3} we obtain
$d_w(u_1,u_2)=0$. Consequently $d_w$ is a pseudometric so, in
accordance with Remark~\ref{r11}, there are no metrics in $\mathfrak
M_w$.
\end{example}
To describe the characteristic structural properties of graphs $G$
having metrics in $\mathfrak M_w$ for each metrizable
$w:H(G)\to\mathbb R^+$ we recall some definitions.

Given an infinite sequence $\{A_n\}_{n\in\mathbb N}$ of sets, we
call the {\it upper limit} of this sequence,
$\limsup_{n\to\infty}A_n$, the set of all elements $a$ such that
$a\in A_n$ holds for an infinity of values of the index $n$. We
have
\begin{equation}\label{x*-1}
\limsup_{n\to\infty}A_n=\bigcap_{k=1}^\infty\bigg(\bigcup_{n=1}^\infty
A_{n+k}\bigg).
\end{equation}

Let $G$ be a graph. A set $F$ of vertices of $G$ is {\it
independent} if every two vertices in $F$ are nonadjacent.

\begin{theorem}\label{t13}
Let $G=(V,E)$ be an infinite connected graph. The following two
conditions are equivalent.
\begin{itemize}
\item[\rm(i)] There is a strictly positive metrizable weight $w$
such that each $\rho\in\mathfrak M_w$ is not  metric but
pseudometric only.

\item[\rm(ii)] There are two  vertices $u^*,v^*\in V(G)$ and a
sequence $\tilde{F}$ of paths $F_j,\ j\in\mathbb N$, joining $u^*$
and $v^*$ such that the upper limit of the sequence
$\{V(F_j)\}_{j\in\mathbb N}$ is an independent set.
\end{itemize}
\end{theorem}
\begin{remark}\label{r3.11}
It is clear that the relations
$$
u^*,v^*\in\limsup_{n\to\infty}V(F_n)
$$
hold for each sequence $\tilde F=\{F_j\}_{j\in\mathbb N}$ joining
$u^*$ and $v^*$. Hence vertices $u^*$ and $v^*$ are nonadjacent if
condition $(ii)$ of Theorem \ref{t13} holds.
\end{remark}
Using \eqref{x*-1} we can reformulate the last part of $(ii)$ in the
form:
\begin{itemize}\it
\item[($ii_1$)] for every $e^0\in E(G)$ there are $u^0\in e^0$ and
$i=i(e^0)$ such that
\begin{equation}\label{e18}
u^0\not\in\bigcup_{k=1}^\infty V(F_{i+k}).
\end{equation}
\end{itemize}

\begin{lemma}\label{l13*}
Let $G$ be an infinite  connected graph, let $u^*$ and $v^*$ be
two distinct nonadjacent vertices of $G$ and let
$\tilde{F}=\{F_j\}_{j\in\mathbb{N}}$ be a sequence of paths
joining $u^*$ and $v^*$ such that $(ii_1)$ holds. Then there is a
subsequence $\{F_{j_k}\}_{k\in\mathbb{N}}$ of $\tilde{F}$ such
that:

$(ii_2)$ the equality
\begin{equation}\label{e3.11}
E(F_{j_l})\cap E(F_{j_k})=\emptyset
\end{equation}
holds whenever $l\ne k$;

$(ii_3)$ if a cycle $C$ is contained in the graph
$\underset{k\in\mathbb{N}}{\bigcup}F_{j_k}$,
\begin{equation}\label{e19}
C\subseteq\underset{k\in\mathbb{N}}{\bigcup}F_{j_k},
\end{equation}
and
\begin{equation}\label{e20*}
k_0=k_0(C):=\min\{k\in\mathbb{N}: E(C)\cap
E(F_{j_k})\neq\emptyset\},
\end{equation}
then $C$ and $F_{j_{k_0}}$ have at least two common edges.
\end{lemma}

\begin{proof}
For every $e\in E(G)$ define a set
$$
N(e):=\{j\in \mathbb N:E(F_j)\ni e\}.
$$
Condition $(ii_1)$ implies that $N(e)$ is finite for each $e\in
E(G)$. Now we construct a subsequence $\{F_{j_k}\}_{k\in\mathbb N}$
by induction. Write $j_1:=1$ and
$$
M_1:=\bigcup_{e\in F_{j_1}}N(e).
$$
Since all $N(e)$ are finite, the set $M_1\subseteq\mathbb N$ is also
finite. Let $j_2$ be the least natural number in the set $\mathbb
N\setminus M_1$. Write
$$
M_2:=\bigcup_{e\in F_{j_2}}N(e),\quad j_3=\min\{m:m\in\mathbb
N\setminus(M_1\cup M_2)\};
$$
$$
M_3:=\bigcup_{e\in F_{j_3}}N(e),\quad j_4=\min\{m:m\in\mathbb
N\setminus(M_1\cup M_2\cup M_3)\}
$$
and so on. It is plain to show that \eqref{e3.11} holds for
distinct $j_k$ and $j_e$. Thus the subsequence
$\{F_{j_k}\}_{k\in\mathbb N}$ satisfies $(ii_2)$. To construct a
subsequence of $\tilde F$ which satisfies simultaneously $(ii_2)$
and $(ii_3)$, note that condition $(ii_1)$ remains valid when one
passes from the sequence $\tilde F$ to any of its subsequences.
Hence, without loss of generality, we may assume that
$\{F_{j_k}\}_{k\in\mathbb N}=\tilde F$.

To define a new subsequence $\{F_{j_k}\}_{k\in\mathbb{N}}$ we again
use the induction. Put $j_1:=1$. Suppose $j_k$ are defined for
$k=1,\ldots,n-1$. By (ii$_1$) for every $e\in F_{j_{n-1}}$ there are
$i(e)\in\mathbb{N}$ and $u\in e$ such that
\begin{equation}\label{e22.0}
u\notin\underset{k=1}{\overset{\infty}{\bigcup}}V(F_{i(e)+k}).
\end{equation}
Define
\begin{equation}\label{e22.1}
j_n:=1+\underset{e\in F_{j_{n-1}}}{\max}i(e).
\end{equation}
Note that $j_n>j_{n-1}$.

Suppose that a cycle $C$ is contained in
$\underset{k\in\mathbb{N}}{\bigcup}F_{j_k}$ where
$\{F_{j_k}\}_{k\in\mathbb{N}}$ is above constructed subsequence of
$\tilde{F}$ and $k_0=k_0(C)$ is defined by \eqref{e20*} but
$F_{j_{k_0}}$ contains the unique edge $e=\{u,v\}$ from $E(C)$. Let
$e_1,e_2\in E(C)$ be the distinct edges which are adjacent to
$e$. The uniqueness of $e$, \eqref{e19} and \eqref{e20*} imply the
relations
\begin{equation}\label{e19*}
e_1\in\underset{k>k_0}{\bigcup}E(F_{j_k})\qquad\text{and} \qquad
e_2\in\underset{k>k_0}{\bigcup}E(F_{j_k}).
\end{equation}
If $e_1=\{u_1,v_1\}$ and $e_2=\{u_2,v_2\}$, then
$$
u\in\{u_1,v_1,u_2,v_2\}\qquad\text{and}\qquad
v\in\{u_1,v_1,u_2,v_2\}
$$
so that from \eqref{e19*} we obtain
$$
u\in\underset{k>k_0}{\bigcup}V(F_{j_k})\qquad\text{and}\qquad
v\in\underset{k>k_0}{\bigcup}V(F_{j_k})
$$
contrary to \eqref{e22.0} and \eqref{e22.1}.
\end{proof}

\begin{proof}[Proof of Theorem \ref{t13}.]
{\bf (i)$\Rightarrow$(ii)} Let $w$ be a strictly positive metrizable
weight such that each $\rho\in\mathfrak M_w$ is a pseudometric only.
Hence $d_w$ is not
 metric, so for some distinct $u^*,v^*\in V(G)$ we have
\begin{equation}\label{e20**}
d_w(u^*,v^*)=0.
\end{equation}
From the definition of $d_w$ it follows at once that there is a
sequence $\tilde F=\{F_i\}_{i\in\mathbb N}$, $F_i\in\mathcal
P_{u^*,v^*}$, $i\in\mathbb N$, such that
\begin{equation}\label{e21}
\lim_{i\to\infty}w(F_i)=0.
\end{equation}
Since  all $F_i$ are finite, we may suppose, taking a subsequence of
$\tilde F$ if it is necessary, that
\begin{equation}\label{e22}
\min_{e\in E(F_j)}w(e)>\sum_{i=1}^\infty w(F_{i+j})
\end{equation}
for all $j\in\mathbb N$. We claim that condition $(ii_1)$ is
fulfilled by $\tilde F$.

Assume there is $e^0=\{u^0,v^0\}\in E(G)$ such that
$$
u^0,v^0\in\bigcup_{k=1}^\infty V(F_{i+k})
$$
for each $i\in\mathbb N$. Since all $F_i$ are paths joining $ u^*$
and $v^*$, there is an $u^0v^0$-walk in the graph
\begin{equation}\label{e24*}
\tilde G_i:=\bigcup_{k=1}^\infty F_{i+k}.
\end{equation}
It is well known that if there is an $xy$-walk in a graph, then
there is also a path joining $x$ and $y$ in the same graph, see, for
example \cite[p.~82]{BM}. Let $P_i$ be a path joining  $u^0$ and
$v^0$ in $\tilde G_i$. The weight $w$ is metrizable. Consequently we
may use Theorem~\ref{t1}. Equalities \eqref{e3}--\eqref{e4} imply
$$
w(\{u^0,v^0\})\leq w(P_i)\leq \sum_{k=1}^\infty w(F_{i+k}).
$$
Letting $i\to\infty$ and using \eqref{e21}, \eqref{e22} we obtain
$$
w(\{u^0,v^0\})\leq\lim_{i\to\infty}\sum_{k=1}^\infty w(F_{i+k})=0,
$$
contrary to the condition: $w(e)>0$ for all $e\in E(G)$.

{\bf(ii)$\Rightarrow$(i)} Let $G$ be a graph satisfying condition
(ii). In view of Lemma \ref{l13*} we may suppose that (ii$_2$) and
(ii$_3$) are also fulfilled with
$\{F_{j_k}\}_{k\in\mathbb{N}}=\tilde{F}$ where $\tilde{F}$ is the
sequence of paths appearing in (ii). To verify condition (i) it
suffices, by Proposition~\ref{p7}, to find a metrizable weight
$w:V(G)\to\mathbb R^+$ such that $w(e)>0$ for all $e\in E(G)$ and
$$
d_w(u^*,v^*)=0
$$
for some distinct vertices $u^*$ and $v^*$.

In the rest of the proof we will construct the desired weight $w$.

Let us consider the graph
$$
\tilde{G}=\underset{i\in\mathbb{N}}{\bigcup}F_i,
$$
cf. \eqref{e24*}. Denote by $m(F_i), \ i\in\mathbb N$, the number of
edges of $F_i$. Let $\{\varepsilon_i\}_{i\in\mathbb
 N}$ be a decreasing sequence of  positive real numbers such that
$\lim_{i\to\infty}\varepsilon_i=0$ and that the sequence
$\{\frac{\varepsilon_i}{m(F_i)}\}_{i\in\mathbb N}$ is also
decreasing. Define a weight $w_1:E(\tilde G)\to\mathbb R^+$ as
\begin{equation}\label{e26}
w_1(e):=\frac{\varepsilon_i}{m(F_i)},\quad\text{if } e\in E(F_i).
\end{equation}
This definition is correct because, by (ii$_2$), the edge sets
$E(F_i)$ and $E(F_j)$ are disjoint for distinct $i$ and $j$.

Let us note that the weight $w_1$ is metrizable. It follows from
Theorem~\ref{t1} because (ii$_3$), \eqref{e26} and the decrease of
the sequence $\{\frac{\varepsilon_i}{m(F_i)}\}_{i\in\mathbb N}$
imply inequality \eqref{e5} for every cycle $C$ in $\tilde G$. (As
has been stated above, \eqref{e5} holds if there are two distinct
edges $e_1,e_2$ in $C$ such that $w(e_1)=w(e_2)=\max_{e\in
E(C)}w(e)$. To find these $e_1$ and $e_2$ we can use (ii$_3$).)

Let $d_{w_1}$ be the weighted shortest-path pseudometric generated
by the weight $w_1$. The condition
$\lim_{i\to\infty}\varepsilon_i=0$ implies $d_{w_1}(u^*,v^*)=0$.
Indeed, we have
\begin{multline*}
d_{w_1}(u^*,v^*)\leq\inf_{i\in\mathbb
N}w_1(F_i)\leq\varlimsup_{i\to\infty}\sum_{e\in E(F_i)}w_1(e)\\
=\lim_{i\to\infty}m(F_i)\frac{\varepsilon_i}{m(F_i)}=\lim_{i\to\infty}\varepsilon_i=0.
\end{multline*}

Let $e^0=\{u^0,v^0\}\in E(G)$ with $u^0,v^0 \in V(\tilde G)$. We
will use (ii$_2$) to prove the inequality
\begin{equation}\label{e27}
d_{w_1}(u^0,v^0)>0.
\end{equation}
Condition (ii$_1$) implies at least one from the relations
$$
u^0\notin\underset{k=1}{\overset{\infty}{\bigcup}}V(F_{i+k})\qquad,\qquad
v^0\notin\underset{k=1}{\overset{\infty}{\bigcup}}V(F_{i+k})
$$
for sufficiently large $i$. Suppose, for instance, that there is
$i_0=i_0(e_0)\in\mathbb N$ such that
\begin{equation}\label{e28}
u^0\not\in V(F_i)\quad \text{if }i>i_0.
\end{equation}
Let $F$ be an arbitrary path in $\tilde G$ joining  $u^0$ and
$v^0$ and let $e\in E(F)$ be the edge incident with the end $u^0$.
From \eqref{e28} follows
$$
e\in\bigcup_{i=1}^{i_0}E(F_i).
$$
Using this relation, \eqref{e26} and the decrease of the sequence
$\{\frac{\varepsilon_i}{m(F_i)}\}_{i\in\mathbb N}$ we obtain
$$
d_{w_1}(u^0,v^0)\geq\frac{\varepsilon_{i_0}}{m(F_{i_0})}>0,
$$
so that \eqref{e27} holds.

Write
$$
V':=V(G)\setminus V(\tilde G).
$$
If $V'=\emptyset$, then the desirable strictly positive weight
$w:E(G)\to\mathbb R^+$ can be obtained as
$$
w(\{u,v\}):=d_{w_1}(u,v),\quad \{u,v\}\in E(G)
$$
because as was shown above $d_{w_1}(u,v)>0$ for each $\{u,v\}\in
E(G)$. The weight $w$ is metrizable because it is a ``restriction''
of the pseudometric $d_{w_1}$. It is easy to show also that
\begin{equation}\label{e29}
d_{w_1}(u,v)=d_w(u,v)
\end{equation}
for all $u,v\in V(G)$. Indeed, since $d_{w_1}\in\mathfrak M_{w}$ the
inequality
$$
d_{w_1}(u,v)\leq d_w(u,v)
$$
follows from Proposition \ref{p7}. To prove the converse inequality
note that
$$
\mathcal P_{u,v}(\tilde G)\subseteq \mathcal P_{u,v}(G)
$$
because $\tilde G$ is a subgraph of $G$. Consequently,
\begin{multline}\label{e29*}
d_w(u,v)=\inf\{w(F):F\in\mathcal
P_{u,v}(G)\}\leq\inf\{w(F):F\in\mathcal P_{u,v}(\tilde G)\}\\
=\inf\{w_1(F):F\in\mathcal P_{u,v}(\tilde G)\}=d_{w_1}(u,v).
\end{multline}
Equality \eqref{e29} implies, in particular, that
$d_w(u^*,v^*)=0$.

Let us consider now the case where
$$
V'=V(G)\setminus V(\tilde G)\ne\emptyset.
$$
Let $v'$ be a fixed point of the set $V'$. Define a distance
function $d'$ on the set $V(G)$ as $d'(v,v)=0$ for all $v\in V(G)$
and as
\begin{equation}\label{e30}
d'(u,v):=\begin{cases} 1&\text{if }u=v^*,\ v=v'\\
1&\text{if }u,v\in V',\ u\ne v\\
d_{w_1}(u,v)&\text{if }u,v\in V(\tilde G)\\
d_{w_1}(u,v^*)+1&\text{if }u\in V(\tilde G),\ v=v'\\
2&\text{if }u=v^*,\ v\in V', \ v\ne v'\\
d_{w_1}(u,v^*)+2&\text{if }u\in V(\tilde G),\ v\in V',\ u\ne u^*,\
v\ne v'.
\end{cases}
\end{equation}
Note that the past three lines in the right side of \eqref{e30} can
be rewritten in the form:
\begin{equation}\label{e30*}
d'(u,v)=d'(u,v^*)+d'(v^*,v')+d'(v',v)
\end{equation}
if $u\in V(\tilde G)$ and $v\in V'$. The last equality and
\eqref{e30} imply
 that $d'$ is a pseudometric. Writing
$$
w(\{u,v\})=d'(u,v)
$$
for all $\{u,v\}\in E(G)$ we obtain the weighted graph $(G,w)$ with
$d'\in\mathfrak M_w$. The weight $w$ is strictly positive since, by
\eqref{e30}, $d'(u,v)\geq1$ if $u\ne v$ and $\{u,v\}\cap
V'\ne\emptyset$ and, by \eqref{e27}, $d'(u,v)>0$ if $u\ne v$ and
$u,v\in V(\tilde G)$, and $\{u,v\}\in E(G)$.

Proposition \ref{p7} implies that
$$
d_w(u,v)=w(\{u,v\})
$$
for each $\{u,v\}\in E(G)$. To complete the proof, it suffices to
observe that $d_w(u^*,v^*)=0$. To see this we can use \eqref{e29*}
with $u=u^*$ and $v=v^*$.
\end{proof}

For the case of disconnected graphs $G$ we have the following
\begin{proposition}\label{p14}
Let $(G,w)$ be a disconnected weighted graph with the strictly
positive metrizable $w$. Then there is a pseudometric
$\rho\in\mathfrak M_w$ which is not  metric. Moreover let $G_i$ be
connected components of $G$ and let $w_i$ be the restrictions of the
weight $w$ on the sets $E(G_i)$. Then the following statements are
equivalent.
\begin{itemize}
\item[\rm(i)] There exists a metric in $\mathfrak M_w$.

\item[\rm(ii)] The shortest-path pseudometrics $d_{w_i}$ are
metrics for all $G_i$.
\end{itemize}
\end{proposition}
\begin{proof}
Set in \eqref{e12} $a_i=0$ for some $i\ne i_0$. Then, by
\eqref{e12*}, $\rho$ is not a metric. If all $d_{w_i}$ are metrics,
then to obtain a metric in $\mathfrak M_w$ it is sufficient to put
$a_i>0$ for all $i\ne i_0$.
\end{proof}

\section{The least element in $\mathfrak M_w$}

We wish characterize the structure of the graphs $G$ for which the
set $\mathfrak M_w$ contains the least pseudometric $\rho_{0,w}$ as
soon as $w$ is metrizable. To this end, we recall the definition of
$k$-partite graph.

\begin{definition}\label{d19}
Let $G$ be a simple graph and let $k$ be a cardinal number. The
graph $G$ is $k$-partite if the vertex set $V(G)$ can be partitioned
into $k$ nonvoid disjoint subsets, or parts, in such a way that no
edge has both ends in the same part. A $k$-partite graph is complete
if any two vertices in different parts are adjacent.
\end{definition}

We shall say that $G$ is a complete multipartite graph if there is a
cardinal number $k\geqslant 1$ such that $G$ is complete
$k$-partite, cf. \cite[p. 14]{Di}.

\begin{remark}\label{r20}
It is easy to prove that each nonempty complete $k$-partite graph
$G$ is connected if $k\geqslant 2$. Each 1-partite graph $G$ is an
empty graph.
\end{remark}

\begin{theorem}\label{t21}
The following conditions are equivalent for each nonempty graph $G$.
\begin{itemize}
\item[\rm(i)] For every metrizable weight
$w:E(G)\to\mathbb{R}^+$ the poset $(\mathfrak{M}_w,\leqslant)$
contains the least pseudometric $\rho_{0,w}$, i.e., the inequality
\begin{equation}\label{e33*}
\rho_{0,w}(u,v)\leqslant\rho(u,v)
\end{equation}
holds for all $\rho\in\mathfrak{M}_w$ and all $u,v\in V(G)$;

\item[\rm(ii)] $G$ is a complete $k$-partite graph with $k\geqslant 2$.
\end{itemize}

If condition (ii) holds and $w$ is a metrizable weight, then for
each pair of distinct nonadjacent vertices $u,v$ we have
\begin{equation}\label{e53}
d_w(u,v)=\underset{
{\substack{\alpha\neq\alpha_0,\\\alpha\in\mathcal{I}}}}{\inf}\,
\underset{p\in X_{\alpha}}{\inf}\bigg(w(\{u,p\})+w(\{p,v\})\bigg),
\end{equation}
and
\begin{equation}\label{e54}
\rho_{0,w}(u,v)=\underset{{\substack{\alpha\neq\alpha_0,\\\alpha\in\mathcal{I}}}}
{\sup}\, \underset{p\in X_{\alpha}}{\sup}|w(\{u,p\})-w(\{p,v\})|
\end{equation}
where $\{X_{\alpha}:\alpha\in\mathcal{I}\}$ is a partition of $G$
and $\alpha_0\in\mathcal{I}$ is the index such that $u,v\in
X_{\alpha_0}$.
\end{theorem}

\begin{remark}
Formulas \eqref{e53} and \eqref{e54} give the generalization of
double inequality \eqref{e2*} for an arbitrary complete $k$-partite
graph, with $k\geqslant 2$. The quadrilateral $Q$ depicted by
Figure~1 is evidently a complete bipartite graph.
\end{remark}

\begin{lemma}\label{l22}
Let $G$ be a connected nonempty graph. The following conditions are
equivalent
\begin{itemize}
\item[\rm(i)] For each metrizable $w:E(G)\to\mathbb{R}^+$ the poset $(\mathfrak M_w,\leqslant)$
contains the least pseudometric $\rho_{0,w}$.

\item[\rm(ii)] For every two distinct nonadjacent vertices $u$ and $v$ and each $p\in
V(G)$ with $u\ne p\ne v$ we have either $$ \{u,p\}\in E(G)\ \&\
\{v,p\}\in E(G) $$ or $$ \{u,p\}\not\in E(G)\ \&\ \{v,p\}\not\in
E(G) .$$
\end{itemize}
\end{lemma}

\begin{proof}
{\bf(i)$\Rightarrow$(ii)} Suppose condition (ii) does not hold. Let
$v_1,v_2,v_3$ be distinct vertices of $G$ such that $v_1$ and $v_2$
are nonadjacent and $v_2$ and $v_3$ are also nonadjacent but
$\{v_3,v_1\}\in E(G)$. Define the weight $w$ as $w(e)=1$ for all
$e\in E(G)$. Let $\rho_1$ and $\rho_2$ be the distance functions on
$V(G)$ such that:
\begin{itemize}
\item[] $\rho_1(v_2,v_1)=\rho_1(v_1,v_2)=\rho_1(u,u)=0$ for all $u\in V(G)$ and
$\rho_1(u,v)=1$ otherwise;

\item[] $\rho_2(v_3,v_1)=\rho_2(v_2,v_3)=\rho_2(u,u)=0$ for all $u\in V(G)$ and
$\rho_2(u,v)=1$ otherwise.
\end{itemize}
It is easy to see that every triangle in $(V(G),\rho_1)$ or in
$(V(G),\rho_2)$ is an isosceles triangle having the third side
shorter or equal to the common length of the other two sides. Hence
$\rho_1$ and $\rho_2$ are pseudometrics and even pseudoultrametrics.
Furthermore $\rho_1$ and $\rho_2$ belong to $\mathfrak M_w$. Suppose
that there is the least pseudometric $\rho_{0,w}$ in $\mathfrak
M_w$. Then we obtain the contradiction
\begin{multline*}
1=\rho_{0,w}(v_2,v_3)\leq\rho_{0,w}(v_2,v_1)+\rho_{0,w}(v_1,v_3)\\\leq
(\rho_1\wedge\rho_2)(v_2,v_1)+(\rho_1\wedge\rho_2)(v_1,v_3)=0+0=0.
\end{multline*}

{\bf(ii)$\Rightarrow$(i)} Suppose condition (ii) holds. Let $w$ be a
metrizable weight. For each pair $u,v$ of vertices of $G$ write:
$$
\rho_0(u,v)=0\quad\text{if } u=v;\qquad \rho_0(u,v)=w\{u,v\}\quad
\text{if }\{u,v\}\in E(G);
$$
and
\begin{equation}\label{e34*}
\rho_0(u,v):=\sup_{P\in\mathcal P_{u,v}}\max_{e\in P}
(2w(e)-w(P))_+
\end{equation}
if  $\{u,v\}\not\in E(G)$ and $ u\ne v$. Here we use the notation
$$
t_+:=\begin{cases} t&\text{if }t\geq0\\ 0&\text{if }t<0.
\end{cases}
$$
We claim that $\rho_0$ is the least element of $(\mathfrak
M_w,\leqslant)$. To show this it suffices to prove the triangle
inequality for $\rho_0$. Indeed, in this case $\rho_0$ is a
pseudometric and $\rho_0\in\mathfrak M_w$ by the definition of
$\rho_0$. Moreover if $\rho$ is an arbitrary pseudometric from
$\mathfrak M_w$, then \eqref{e5} implies:
$$
2w(e)\leq w(P)+\rho(u,v)
$$
for all distinct $u,v\in V(G)$,  all $P\in\mathcal P_{u,v}$ and
all $e\in P$. Consequently we have
$$
2w(e)-w(P)\leq \rho(u,v),
$$
$$
(2w(e)-w(P))_+\leq (\rho(u,v))_+=\rho(u,v),
$$
\begin{equation}\label{e35}
\sup_{P\in\mathcal P_{u,v}}\max_{e\in
P}(2w(e)-w(P))_+\leq\rho(u,v).
\end{equation}
The last inequality and \eqref{e34*} imply \eqref{e33*} with
$\rho_{0,w}=\rho_0$. Thus $\rho_0$ is the least pseudometric in
$\mathfrak M_w$.

Let us turn to the triangle inequality for $\rho_0$. Let $x,y,z$ be
some distinct vertices of $G$. Since $w$ is metrizable, the
definition of $\rho_0$ implies this inequality if $\{x,y\},\{y,z\}$
and $\{z,x\}$ belong to $E(G)$. Let $\{x,y\}\not\in E(G)$. In
accordance with condition (ii), either both $\{y,z\}$ and $\{z,x\}$
are   edges of $G$ or both $\{y,z\}$ and $\{z,x\}$ are not  edges of
$G$.

Suppose $\{y,z\},\{z,x\}\in E(G)$. The three-point sequence
$P_1:=(x,z,y)$ is a path joining   $x$ and $y$. Consequently by
\eqref{e34*} we obtain
$$
\rho_0(x,y)\geq\max_{e\in
P_1}(2w(e)-w(P_1))_+=|w(\{x,z\})-w(\{z,y\})|.
$$
Thus
\begin{equation}\label{e36}
\rho_0(x,y)+\min(\rho_0(x,z),\rho_0(z,y))\geq
\max(\rho_0(x,z),\rho_0(z,y)).
\end{equation}
To prove the inequality
$$
\rho_0(x,y)\leq\rho_0(x,z)+\rho_0(z,y)
$$
it is sufficient to show
\begin{equation}\label{e37}
\max_{e\in P}(2w(e)-w(P))_+\leq\rho_0(x,z)+\rho_0(z,y)
\end{equation}
for each path $P$ joining  $x$ and $y$. This inequality is trivial
if its left part equals zero. In the opposite case, \eqref{e37}
can be rewritten in the form
$$
2\max_{e\in P}w(e)\leq w(P)+w(\{x,z\})+w(\{z,y\}).
$$
Applying \eqref{e5} we see that the last inequality holds, so
\eqref{e37} follows. (Note that inequality \eqref{e5} holds for
each closed walk in $G$ if it holds for each cycle in $G$.)

It is slightly more difficult to prove the triangle inequality for
$\rho_0$ when
\begin{equation}\label{e38}
\{y,z\}\not\in E(G),\quad \{z,x\}\not\in E(G)\quad \text{and}\quad
\{z,y\}\not\in E(G).
\end{equation}
To this end, we establish first the following lemma.

\begin{lemma}\label{l18}
Let $(G,w)$ be a connected, weighted graph with a metrizable $w$,
let condition {\rm (ii)} of Lemma \ref{l22} hold and let $x,y$ be
distinct nonadjacent vertices of $G$. Then, for every $P\in\mathcal
P_{x,y}$ there is $v\in V(G)$ with $\{v,x\},\{v,y\}\in E(G)$ and
such that
\begin{equation}\label{e39}
\max_{e\in P}(2w(e)-w(P))_+\leq |w(\{x,v\})-w(\{v,y\})|.
\end{equation}
\end{lemma}
\begin{proof}
Let $P=(x,v_1,\dots,v_n,y)$ be a path joining  $x$ and $y$ in $G$.
We claim that there is a path $(x,v,y)$ in $G$ such that
\eqref{e39} holds. It is trivial  if $n=1$ or if the left part in
\eqref{e39} is zero. So we may suppose that $n\geq2\ (v_1\ne v_2)$
and
\begin{equation}\label{e40}
2\max_{e\in P}w(e)>w(P).
\end{equation}
Condition (ii) of Lemma~\ref{l22} implies
\begin{equation}\label{e41}
\{v_1,y\}\in E(G)\quad\text{and}\quad \{x,v_n\}\in E(G),
\end{equation}
see Figure 4. For convenience we write
$$
M:=\max_{e\in P}w(e).
$$\begin{figure}[h]
\centering
\includegraphics[keepaspectratio,width=10cm]{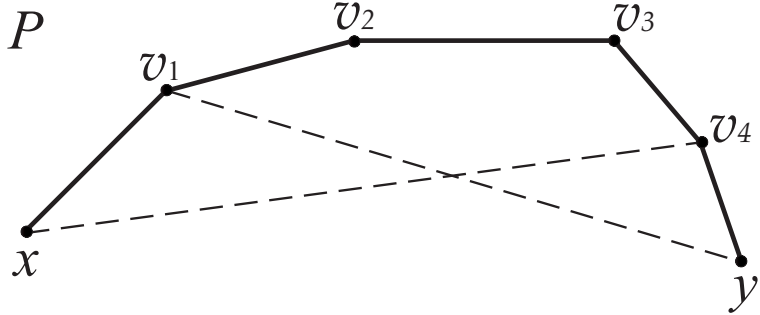}
\caption{The path $P$ joining   $x$ and $y$ with $n=4$ (heavily
drawn lines), and two additional edges $\{x,v_4\},\{y,v_1\}$
(dotted lines).}
\end{figure}

\noindent Let us prove \eqref{e39}. If $M=w(\{x,v_1\})$, then
\begin{multline*}
2M-w(P)=w(\{x,v_1\})-\bigg(\sum_{i=1}^{n-1}
w(\{v_i,v_{i+1}\})+w(\{v_n,y\})\bigg)\\\leq
w(\{x,v_1\})-w(\{v_1,y\})
\end{multline*}
because $w$ is metrizable and so  we have the ``triangle
inequality''
$$
w(\{v_1,y\})\leq \sum_{i=1}^{n-1} w(\{v_i,v_{i+1}\})+w(\{v_n,y\}).
$$
Hence the path $(x,v_1,y)$ satisfies \eqref{e39} with $v=v_1$.
Similarly if $M=(\{v_n,y\})$, then the desired path is
$(x,v_n,y)$.

Suppose now that
\begin{equation}\label{e41}
w(\{v_1,y\})\geq w(\{x,v_1\})\quad\text{and}\quad M=\max_{1\leq
i\leq n-1}w(\{v_i,v_{i+1}\}).
\end{equation}
Since $w$ is metrizable, applying \eqref{e5} to the cycle
$(v_1,v_2,\dots,y,v_1)$ we obtain
$$
w(\{v_1,y\})+w(\{v_n,y\})+\sum_{i=1}^{n-1}w(\{v_i,v_{i+1}\})\geq
2M.
$$
Consequently
$$
w(\{v_1,y\})-w(\{x,v_1\})\geq 2M-w(P).
$$
Thus $(x,v_1,y)$ satisfies \eqref{e39} with $v=v_1$ if \eqref{e41}
holds. Similarly \eqref{e39} holds with $v=v_n$ if
\begin{equation}\label{e42}
w(\{v_n,x\})\geq w(\{y,v_n\})\quad\text{and}\quad M=\max_{1\leq
i\leq n-1}w(\{v_i,v_{i+1}\}).
\end{equation}

It still remains to find $(x,v,y)$ satisfying \eqref{e39} if
\begin{equation}\label{e43}
\begin{gathered}
w(\{v_1,y\})\leq w(\{x,v_1\}),\quad w(\{v_n,x\})\leq
w(\{y,v_n\})\\
\text{and } M=\max_{1\leq i\leq n-1}w(\{v_i,v_{i+1}\}).
\end{gathered}
\end{equation}

Let us consider the new path $F=(x,u_1,\dots,u_n,y)$ such that
$u_1=v_n,\ u_2=v_{n-1},\dots, u_n=v_1$, see Fig.~5.
\begin{figure}[h]
\centering
\includegraphics[keepaspectratio,width=10cm]{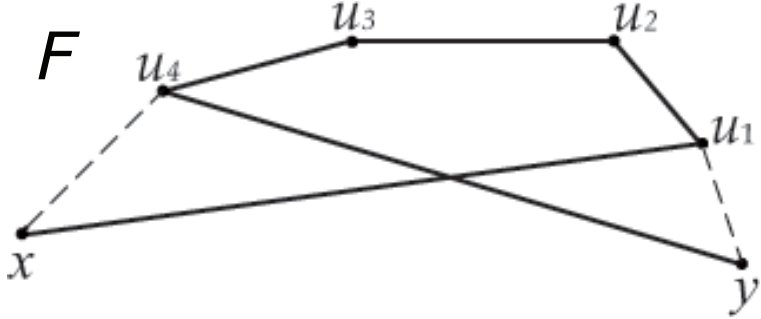}
\caption{The new path $F$ is a modification of the old path $P$.}
\end{figure}
Condition \eqref{e43} implies that
$$
M=\max_{1\leq i\leq n-1} w(\{u_i,u_{i+1}\})=\max_{e\in F}w(e)
$$
and, moreover, that $w(F)\leq w(P)$. Hence it suffices to prove
the inequality
$$
\max_{e\in F}(2w(e)-w(F))_+\leq|w(\{x,v\})-w(\{v,y\})|
$$
for a 2-path $(x,v,y)$ in $G$. We can make it in a way analogous to
that was used under consideration of restriction \eqref{e41} if
$$
w(\{u_1,y\})\geq w(\{x,u_1\}).
$$
To complete the proof, it suffices to observe that the last
inequality can be rewritten as
$$
w(\{v_n,y\})\geq w(\{x,v_n\})
$$
which follows from \eqref{e43}.
\end{proof}

{\it Continuation of the proof of Lemma \ref{l22}.} It still remains
to prove the inequality
\begin{equation}\label{e44}
\rho_0(x,y)\leq \rho_0(x,z)+\rho_0(z,y)
\end{equation}
if $x,y,z$ are distinct vertices such that
\begin{equation}\label{e45}
\{x,y\}\not\in E(G),\quad \{x,z\}\not\in
E(G)\quad\text{and}\quad\{z,y\}\not\in E(G).
\end{equation}
It follows from Lemma \ref{l18} that
\begin{equation}\label{e46}
\rho_0(x,y)=\sup_v|w(\{x,v\})-w(\{v,y\})|
\end{equation}
where the supremum is taken over the set of all vertices $v$ such
that
\begin{equation}\label{e47}
\{x,v\},\{v,y\}\in E(G).
\end{equation}
Condition (ii), relations \eqref{e45} and relations \eqref{e47}
give the membership relation
$$
\{z,v\}\in E(G).
$$
Thus the weight function $w$ is defined at the ``point''
$\{z,v\}$. Hence
\begin{multline*}
|w(\{x,v\})-w(\{v,y\})|\leq
|w(\{x,v\})-w(\{v,z\})|+|w(\{v,z\})-w(\{v,y\})|\\\leq\rho_0(x,z)+\rho_0(z,y).
\end{multline*}
These inequalities and \eqref{e46} imply \eqref{e44}.
\end{proof}

Recall that a subgraph $H$ of a graph $G$ is {\it induced} if $E(H)$
consists of all edges of $G$ which have both ends in $V(H)$.

\begin{figure}[h]
\includegraphics[width=\textwidth,keepaspectratio]{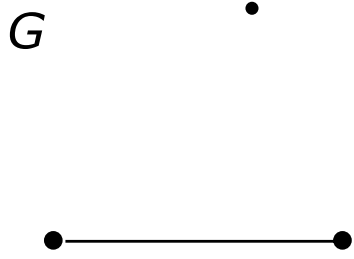}
\caption{ The graph $G$ is not complete $k$-partite for any
$k\geqslant 1$. }
\end{figure}

\begin{proposition}\label{p26*}
A nonnull graph is complete multipartite if and only if it has no
induced subgraphs depicted by Figure~6.
\end{proposition}

\begin{proof}
Suppose $G$ is a complete $k$-partite graph. If $k=1$, then all
subgraphs of $G$ are empty. Let $k\geqslant 2$. If $u$ and $v$ are
vertices of $G$ such that $\{u,v\}\notin E(G)$, then there exists a
part $V_1$ in the partition of $V(G)$ such that
$$
u\in V_1\quad\text{ and }\quad v\in V_1.
$$
If $p$ is a vertex of $G$ and $u\neq p\neq v$, then either $p\in
V_1$ or there is a part $V_2\neq V_1$ such that $p\in V_2$. Using
Definition \ref{d19} we obtain that $\{p,v\}\notin E(G)$ and
$\{p,u\}\notin E(G)$ if $p\in V_1$ or, in the opposite case $p\in
V_2$, that $\{p,v\}\in E(G)$ and $\{p,u\}\in E(G)$.

Assume now that $G$ has no induced subgraphs depicted by Figure~6.
Let us define a relation $\asymp$ on the set $V(G)$ as
\begin{equation}\label{e49}
(u\asymp v) \Leftrightarrow (\{u,v\}\notin E(G)).
\end{equation}
Relation $\asymp$ is evidently symmetric. Since simple graphs
contain no loops, we have $\{u,u\}\notin E(G)$ for each $u\in V(G)$.
Consequently $\asymp$ is reflexive. Moreover if $\{u,v\}\notin E(G)$
and $\{v,p\}\notin E(G)$, then we obtain $\{u,p\}\notin E(G)$. Thus
$\asymp$ is transitive, so this is an equivalence relation. The set
$V(G)$ is partitioned by the relation $\asymp$ on the disjoint parts
$V_i$, $i\in \mathcal{I}$, where $\mathcal{I}$ is an index set. It
follows directly from \eqref{e49} that no edge of $G$ has both ends
in the same part. Hence $G$ is a $k$-partite graph with
$k=\card\mathcal{I}$. Finally note that $\{u,v\}\in E(G)$ if and
only if the relation $u \asymp v$ does not hold. Consequently $G$ is
a complete $k$-partite graph.
\end{proof}

This proposition implies

\begin{lemma}\label{l24}
Let $G$ be a nonempty graph. Condition (ii) of Lemma \ref{l22} holds
if and only if $G$ is a complete $k$-partite graph with $k\geqslant
2$.
\end{lemma}

It still remains to prove the next lemma.

\begin{lemma}\label{l25}
Let $G$ be a nonempty graph. If condition (i) of Theorem \ref{t21}
holds, then $G$ is connected.
\end{lemma}

\begin{proof}
Let $w:E(G)\to\mathbb{R}^+$ be a weight such that the equality
\begin{equation}\label{e50}
w(e)=1
\end{equation}
holds for all $e\in E(G)$. It is clear that $w$ is metrizable. Let
$\{u_1,v_1\}$ be an edge of $G$. If $G$ is disconnected, then there
are two connected components $G_1$ and $G_2$ of $G$ such that
$$
u_1\in V(G_1), \quad v_1\in V(G_1)\quad\text{ and }\quad V(G_1)\cap
V(G_2)=\emptyset.
$$
Let $p$ be a vertex of $G_2$. Using formulas \eqref{e11},
\eqref{e12} with zero constants $a_i$ we can find some pseudometrics
$\rho_1,\rho_2\in\mathfrak M_w$ for which
\begin{equation}\label{e51}
\rho_1(u_1,p)=0\qquad\text{ and }\qquad\rho_2(v_1,p)=0.
\end{equation}
If condition (i) of Theorem \ref{t21} holds, then for the least
pseudometric $\rho_{0,w}$ in $\mathfrak M_w$ we have the
inequalities
$$
\rho_{0,w}(u_1,p)\leqslant\rho_1(u_1,p)\quad\text{ and
}\quad\rho_{0,w}(v_1,p)\leqslant\rho_2(v_1,p).
$$
These inequalities, the triangle inequality and \eqref{e51} imply
$$
\rho_{0,w}(u_1,v_1)\leqslant\rho_1(u_1,p)+\rho_2(v_1,p)=0.
$$
Since $\rho_{0,w}\in\mathfrak M_w$, it implies $w(e)=0$ for
$e=\{u_1,v_1\}$, contrary \eqref{e50}.
\end{proof}

\begin{proof}[Proof of Theorem \ref{t21}.]
Suppose that condition (i) of the theorem holds. By Lemma \ref{l25},
$G$ is a connected graph and so we can use Lemma \ref{l22}. Applying
this lemma we obtain the equivalence of its condition (ii) with
condition (i) of Theorem \ref{t21}. By Lemma \ref{l24} condition
(ii) of Lemma \ref{l22} implies condition (ii) of the Theorem
\ref{t21}.

Conversely, suppose that condition (ii) of Theorem \ref{t21} holds.
Using Lemma \ref{l24} we see that condition (ii) of Lemma \ref{l22}
holds. Moreover condition (ii) of Theorem \ref{t21} implies that $G$
is connected, see Remark \ref{r20}. Hence by Lemma \ref{l22} we
obtain condition (i) of the theorem. Thus we have the implication
(ii)$\Rightarrow$(i) in Theorem \ref{t21}.

Assume now that $G$ is complete $k$-partite graph with $k\geqslant
2$ and $w$ is metrizable. Let $u$ and $v$ be some distinct
nonadjacent vertices of $G$. Then we have $u,v\in X_{\alpha_0}$ for
some $\alpha_0\in\mathcal{I}$. We must to prove equalities
\eqref{e53} and \eqref{e54}. For every vertex $p\notin X_{\alpha_0}$
the sequence $(u,p,v)$ is a path joining $u$ and $v$. Consequently
the inequality
\begin{equation}\label{e58}
d_w(u,v)\leqslant\underset{
{\substack{\alpha\neq\alpha_0,\\\alpha\in\mathcal{I}}}}{\inf}\,
\underset{p\in X_{\alpha}}{\inf}|w(\{u,p\})+w(\{p,v\})|
\end{equation}
follows from \eqref{e3}. To prove the converse inequality it is
sufficient to show that for every $F\in\mathcal{P}_{u,v}$ there is
$p\in X_{\alpha}$, $\alpha\neq\alpha_0$, such that
\begin{equation}\label{e58*}
w(F)\geqslant w(\{u,v\})+w(\{p,v\}).
\end{equation}
Since $u$ and $v$ are nonadjacent, the length (the number of edges)
of $F$ is more or equal 2 for every $F\in\mathcal{P}_{u,v}$.  If the
length of $F$ is 2, then the "inner" vertex of $F$ does not belong
to $X_{\alpha_0}$ so we have \eqref{e58}. Let
$(u=v_0,v_1,\ldots,v_n=v)$ belong to $\mathcal{P}_{u,v}$ and
$n\geqslant 3$. If $v_1\in X_{\alpha}$, then $\alpha\neq\alpha_0$
and $\{u,v\}\in E(G)$ because $G$ is a complete $k$-partite graph.
Since $w$ is metrizable, statement (ii) of Theorem \ref{t1} implies
\begin{equation}\label{e59}
w(\{v_1,v\})\leqslant w(F')
\end{equation}
where $F'$ is the part $(v_1,\ldots,v_n)$. It is clear that
$$
w(F)=w(\{u,v_1\})+w(F').
$$
Consequently \eqref{e59} implies \eqref{e58*} with $p=v_2$. Equality
\eqref{e53} follows.

To prove \eqref{e54} we now return to lemmas \ref{l22}, \ref{l18},
\ref{l24}. By the assumption $G$ is a complete $k$-partite graph.
Hence, by Lemma \ref{l24}, condition (ii) of Lemma \ref{l22} holds.
This condition implies that
\begin{equation}\label{e60}
\rho_{0,w}(u,v)=\underset{F\in\mathcal{P}_{u,v}}{\sup}
\underset{e\in F}{\max}(2w(e)-w(F))_+
\end{equation}
see \eqref{e34*} and \eqref{e35}. Using Lemma \ref{l18} we obtain
that for every $F\in\mathcal{P}_{u,v}$ there is $p\in V(G)$ with
$\{u,p\},\{p,v\}\in E(G)$ and such that
$$
\underset{e\in F}{\max}(2w(e)-w(F))_+\leqslant
|w(\{u,p\})-w(\{p,v\})|.
$$
Consequently we have
$$
\rho_{0,w}(u,v)\leqslant\underset{
{\substack{\alpha\neq\alpha_0,\\\alpha\in\mathcal{I}}}}{\sup}\,
\underset{p\in X_{\alpha}}{\sup}|w(\{u,p\})-w(\{p,v\})|
$$
for every two distinct nonadjacent vertices $u,v$. The converse
inequality follows from \eqref{e60}. Indeed, for every path $F$ of
the form $(u,p,v)$ we have
$$
|w(\{u,p\})-w(\{p,v\})|=\underset{e\in F}{\max}(2w(e)-w(F))_+.
$$
\end{proof}

Recall that the star is a complete bipartite graph $G$ with a
bipartition $(X,Y)$,
$$
V(G)=X\cup Y,\qquad X\cap Y=\emptyset
$$
such that $\card X=1$ or $\card Y=1$.

\begin{corollary}
The following conditions are equivalent for each nonempty graph $G$.
\begin{itemize}
\item[\rm(i)] Every weight $w:E(G)\to\mathbb{R}^+$ is
metrizable and the poset $(\mathfrak{M}_w,\leqslant)$ contains the
least pseudometric $\rho_{0,w}$.

\item[\rm(ii)] $G$ is a star.
\end{itemize}
\end{corollary}
\begin{proof}
Let condition (i) hold. Then, by Theorem \ref{t21},  $G$ is complete
$k$-partite  with $k\geqslant 2$ and by Corollary \ref{c6} $G$ is
acyclic. Each $k$-partite graph with $k\geqslant 3$ contains a
3-cycle (triangle). Hence we have $k=2$, i.e. $G$ is bipartite. If
$(X,Y)$ is a bipartion of $G$ with
$$
\card X\geqslant 2\qquad\text{and}\qquad\card Y\geqslant 2,
$$
then we can find some vertices
$$
x_1,x_3\in X\qquad\text{and}\qquad x_2,x_4\in Y.
$$
Since $G$ is a complete bipartite graph, $G$ contains the
quadrilateral $Q$, see Fig.~1. Consequently we have $\card X=1$ or
$\card Y=1$. Thus $G$ is a star.

Conversely suppose $G$ is a star. Then $G$ is acyclic, so using
Corollary \ref{c6} we obtain that every weight $w$ is metrizable.
Since stars are complete bipartite graphs, Theorem \ref{t21} implies
the existence of the least pseudometric $\rho_{0,w}\in\mathfrak
M_w$.
\end{proof}

\begin{theorem}\label{t40}
The following conditions (i) and (ii) are equivalent for each
nonempty graph $G$.
\begin{itemize}
\item[\rm(i)] For each metrizable weight $w:E(G)\to\mathbb{R}^+$ the
set $\mathfrak{M}_w$ contains the least pseudometric $\rho_{0,w}$
and this set contains also all symmetric functions $f:V(G)\times
V(G)\to \mathbb{R}^+$ which lie between $\rho_{0,w}$ and the
shortest-path pseudometric $d_w$, i.e., which satisfy the double
inequality
\begin{equation}\label{e52}
\rho_{0,w}(u,v)\leqslant f(u,v)\leqslant d_w(u,v)
\end{equation}
for all $u,v\in V(G)$.

\item[\rm(ii)] $G$ is a complete $k$-partite graph with a partition
$\{X_{\alpha} : \alpha\in\mathcal{I}\}$ such that
$\card\mathcal{I}=k\geqslant 2$ and $\card X_{\alpha}\leqslant 2$
for each part $X_{\alpha}$.
\end{itemize}
\end{theorem}

\begin{proof}
{\bf(i)$\Rightarrow$(ii)} Suppose that (i) holds. By Theorem
\ref{t21} $G$ is a complete $k$-partite graph with $k\geqslant 2$.
Assume that there is a part $X_{\alpha_0}$ such that
$\card(X_{\alpha_0})\geqslant 3$. Let $v_1,v_2,v_3$ be some pairwise
distinct elements of $X_{\alpha_0}$ and let $w$ be the weight such
that $w(e)=1$ for each $e\in E(G)$. Define functions $\rho_1$ and
$\rho_2$ on the set $V(G)\times V(G)$ as
$$
\rho_1(u,v)=\begin{cases} 0\quad\text{if } u=v\\
1\quad\text{if } u\neq v
\end{cases}
\text{and}\quad
\rho_2(u,v)=\begin{cases} 1\quad\text{if } \{u,v\}\in E(G)\\
0\quad\text{if } \{u,v\}\notin E(G)
\end{cases}.
$$
It is clear that $\rho_1\in\mathfrak{M}_w$. To prove that
$\rho_2\in\mathfrak{M}_w$ it is sufficient to verify the triangle
inequality
\begin{equation}\label{e55}
\rho_2(u,v)\leqslant\rho_2(u,s)+\rho_2(v,s)
\end{equation}
for all $u,v,s\in V(G)$. If \eqref{e55} does not hold, then
$\rho_2(u,v)=1$ and $\rho_2(u,s)=\rho_2(v,s)=0$. Consequently we
have that
\begin{equation}\label{e56}
\{u,v\}\in E(G)\text{ and }\{u,s\}\notin E(G)\text{ and
}\{v,s\}\notin E(G).
\end{equation}

The relation $\{u,s\}\notin E(G)\text{ and }\{v,s\}\notin E(G)$
imply that $u,s$ belong to a part $X_u$, similarly $v,s$ belong to a
part $X_v$. Since $s\in X_u\cap X_v$ we obtain that $X_u=X_v$, hence
$\{u,v\}\notin E(G)$ contrary to the first membership relation in
\eqref{e56}. Thus \eqref{e55} holds for all $u,v,s\in V(G)$, so
$\rho_2\in\mathfrak{M}_w$. The function $f:V(G)\times
V(G)\to\mathbb{R}^+$ defined as
$$
f(v_1,v_2)=f(v_2,v_1)=1,\quad
f(v_1,v_3)=f(v_3,v_1)=f(v_3,v_2)=f(v_2,v_3)=0
$$
and
$$
f(u,v)=\rho_1(u,v)
$$
for $(u,v)\in(V(G)\times
V(G))\setminus\{(v_1,v_2),(v_2,v_1),(v_2,v_3),(v_3,v_2),(v_1,v_3),(v_3,v_2)\}$
satisfies the double inequality
$$
\rho_2(u,v)\leqslant f(u,v)\leqslant\rho_1(u,v)
$$
that implies \eqref{e52}. Hence by (i) we must have
$$
f(v_1,v_2)\leqslant f(v_1,v_3)+f(v_3,v_2)
$$
that contradicts the definition of the function $f$. Thus the
inequality $\card X_{\alpha}\break\leqslant 2$ holds for each part
$X_{\alpha}$.
\begin{figure}[h]
\includegraphics[width=12cm,keepaspectratio]{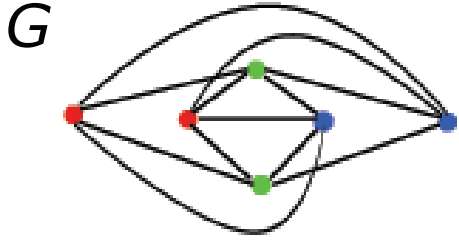}
\caption{A complete 3-partite graph $G$ satisfies condition (ii)
of Theorem~\ref{t40}.}
\end{figure}

{\bf(ii)$\Rightarrow$(i)} Suppose that condition (ii) holds. Since
$k\geqslant 2$ and  $G$ is a complete $k$-partite graph, Theorem
\ref{t21} provides the existence of the least pseudometric
$\rho_{0,w}$ for each metrizable weight $w$. Let $f:V(G)\times
V(G)\to\mathbb{R}^+$ be a symmetric function such that \eqref{e52}
holds for all $u,v\in V(G)$. The double inequality \eqref{e52}
implies that $f$ is nonnegative and $f(u,v)=w(\{u,v\})$ for all
$\{u,v\}\in E(V)$ and $f(u,u)=0$ for all $u\in V(G)$. Consequently
to prove that $f\in\mathfrak{M}_w$ it is sufficient  to obtain the
triangle inequality
\begin{equation}\label{e57}
f(u,v)\leqslant f(u,p)+f(p,v)
\end{equation}
for all $u,v,p\in V(G)$. We may assume $u,v$ and $p$ are pairwise
disjoint, otherwise \eqref{e57} is trivial. Since $\card
X_{\alpha}\leqslant 2$ for each part $X_{\alpha}$, at most one pair
from the vertices $u,v$ and $p$ are nonadjacent. If $\{u,v\}\notin
E(G)$, then using \eqref{e52} we obtain
\begin{multline*}
f(u,v)\leqslant d_w(u,v)\leqslant d_w(u,p)+d_w(p,v)\\=
w(\{u,p\})+w(\{p,v\})\leqslant f(u,p)+f(p,v).
\end{multline*}
Similarly if $\{u,p\}\notin E(G)$ or $\{p,v\}\notin E(G)$, then we
have
$$
f(u,v)\leqslant \rho_{0,w}(u,v)\leqslant
\rho_{0,w}(u,p)+\rho_{0,w}(p,v)= f(u,p)+f(p,v).
$$
Inequality \eqref{e57} follows and we obtain condition (i).
\end{proof}

\medskip{\bf Acknowledgment.} The first author is thankful to the 
Finnish Academy of Science and Letters for the support. 
The research of M. Vuorinen was supported, 
in part, by the Academy of Finland, Project 2600066611.


\vskip8mm

{\bf Oleksiy Dovgoshey}

{\small  Institute of Applied Mathematics and Mechanics of NASU,}

{\small  R.Luxemburg str. 74, Donetsk 83114, Ukraine;}

{\small E-mail: {\it aleksdov@mail.ru}

 \vskip8mm

{\bf Olli Martio}

{\small  Department of Mathematics and Statistics,}

{\small  University of Helsinki,}

{\small  P.O. Box 68 FI-00014 University of Helsinki, Finland;}

{\small E-mail: {\it olli.martio@helsinki.fi}}

\vskip8mm

{\bf Matti Vuorinen}

{\small  Department of Mathematics,}

{\small  University of Turku,}

{\small  FIN-20014, Turku, Finland;}

{\small E-mail: {\it vuorinen@utu.fi}}

\end{document}